\pdfoutput=1
\documentclass[pdftex,12pt,a4paper]{article} 

\usepackage{xfrac,amsmath,amssymb,amsthm}
\usepackage[pdfpagemode=None,colorlinks=true,linkcolor=blue,citecolor=blue]{hyperref}

\usepackage{fullpage}
\usepackage{microtype}
\usepackage{fourier}

\usepackage{tikz}
\usetikzlibrary{calc}
\tikzset{outline/.style={color=black,thick}}
\tikzset{vertex/.style={circle,inner sep=0pt,minimum size=12mm,draw=black,fill=white}}

\usepackage{array}
\newcolumntype{y}[1]{>{\centering\arraybackslash$\displaystyle} p{#1} <{$}}   
\newcolumntype{x}[1]{>{\centering\arraybackslash\hspace{0pt}}p{#1}}

\newlength{\mycolumnwidth}
\usepackage{enumerate}

\newcommand\ric{\mathop{ric}}	
\newcommand{\mathd}{\mathrm{d}}

\newcounter{thm}
\newcounter{ex}
\newcounter{re}
\newtheorem{Theorem}[thm]{Theorem}
\newtheorem{Lemma}[thm]{Lemma}
\newtheorem{Proposition}[thm]{Proposition}
\newtheorem{Remark}[re]{Remark}

\title{Ricci curvature on polyhedral surfaces via optimal transportation}
\author{Beno\^it Loisel and Pascal Romon}
\date{}

\begin{document}

\maketitle

\begin{abstract}
  The problem of defining correctly geometric objects such as the curvature is
  a hard one in discrete geometry. In 2009, Ollivier \cite{Ollivier:2009p2415} defined 
  a notion of curvature applicable to a wide category of measured
  metric spaces, in particular to graphs. He named it {\emph{coarse Ricci curvature}} 
  because it coincides, up to some given factor, with the classical Ricci
  curvature, when the space is a smooth manifold. Lin, Lu \& Yau \cite{Lin:2011tt}, 
  Jost \& Liu \cite{Jost:2011wt} have used and extended this notion for graphs giving 
  estimates for the curvature and hence the diameter, in terms of the combinatorics. 
  In this paper, we describe a method for computing the coarse Ricci curvature
  and give sharper results, in the specific but crucial case of polyhedral surfaces.
  \\
  
  \\
  

\end{abstract}

\section{The coarse Ricci curvature of Ollivier}

Let us first recall the definition of the coarse Ricci curvature as given
originally by Ollivier in \cite{Ollivier:2009p2415}.
Since our focus is on polyhedral objects, we will use, for greater legibility,
the language of graphs and matrices, rather than more general measure
theoretic formulations. In this section, we will assume that our base space $X
= ( V, E)$ is a simple graph (no loops, no multiple edges between vertices),
unoriented and locally finite (vertices at finite distance are in finite
number) for the standard distance (the length of the shortest chain between
$x$ and $x'$):
\[ 
\forall x, x' \in V, \; d ( x, x') = \inf \left\{ n, \; n \in
   \mathbb{N}^{\ast}, \; \exists x_1, \ldots, x_{n - 1}, \; x = x_0 \sim x_1
   \sim \cdots \sim x_n = x' \right\} 
\]
where $x \sim y$ stands for the adjacency relation and adjacent vertices are
at distance $1$. We shall call this distance the uniform distance; the more
general case will be considered in section~\ref{sec:varying-length}. 
Polyhedral surfaces (a.k.a. two-dimensional cell complexes) studied afterwards will be seen 
as a special case of graphs. We do not require our graphs to be finite, but we will
nevertheless use vector and matrix notations (with possibly infinitely many
indices).
\\

For any $x \in X$, a probability measure $\mu$ on $X$ is a map $V \rightarrow
\mathbb{R}_+$ such that \hbox{$\sum_{y \in V} \mu ( y) = 1$}. Assume that for
any vertex $x$ we are given a probability measure $\mu_x$. Intuitively $\mu_x (y)$ 
is the probability of jumping from $x$ to $y$ in a random walk. For instance we can
take $\mu_x^1$ to be the uniform measure on the sphere (or 1-ring) at $x$,
$S_x = \{ y \in V,  y \sim x \}$, namely $\mu^1_x ( y) = \frac{1}{d_x}$
if $y \sim x$, where $d_x = | S_x |$ denotes the degree of $x$, and $0$
elsewhere. A {\emph{coupling}} or {\emph{transference plan}} $\xi$ between
$\mu$ and $\mu'$ is a measure on $X \times X$ whose marginals are $\mu, \mu'$
respectively:
\[ \sum_{y'} \xi ( y, y') = \mu ( y) \text{ and } \sum_y \xi ( y, y') = \mu'
   ( y') . \]
Intuitively a coupling is a plan for transporting a mass $1$ distributed
according to $\mu$ to the same mass distributed according to $\mu'$. Therefore
$\xi ( y, y')$ indicates which quantity is taken from $y$ and sent to $y'$.
Because mass is nonnegative, one can only take from $x$ the quantity $\mu (
x)$, no more no less, and the same holds at the destination point, ruled by
$\mu'$. We may view measures as vectors indexed by $V$ and couplings as
matrices, and we will use that point of view later.

The {\emph{cost}} of a given coupling $\xi$ is
\[ 
c(\xi) = \sum_{y, y' \in V} \xi ( y, y') d ( y, y') 
\]
where cost is induced by the distance traveled. The Wasserstein distance $W_1$
between probability measures $\mu, \mu'$ is
\[ W_1 ( \mu, \mu') = \inf_{\xi}  \sum_{y, y' \in V} \xi ( y, y') d ( y, y')
\]
where the infimum is taken over all couplings $\xi$ between $\mu$ and $\mu'$
(such a set will never be empty and we will show that its infimum is attained
later on). Let us focus on two simple but important examples:
\begin{enumerate}
  \item Let $\mu_x$ be the Dirac measure $\delta_x$, i.e. $\delta_x ( y)$
  equals $1$ when $y = x$ and zero elsewhere. Then there is only one coupling
  $\xi$ between $\delta_x$ and $\delta_{x'}$ and it satisfies $\xi ( y, z) =
  1$ if $y = x$ and $z = x'$, and vanishes elsewhere. Obviously $W_1 ( x, x')
  = d ( x, x')$.
  
  \item Consider now $\mu^1_x, \mu^1_{x'}$ the uniform measures on unit
  spheres around $x$ and $x'$ respectively; then a coupling $\xi$ vanishes on
  $( y, y')$ whenever $y$ lies outside the sphere $S_x$ or $y'$ outside
  $S_{x'}$. So the (a priori infinite) matrix $\xi$ has at most $d_x d_{x'}$
  nonzero terms, and we can focus on the $d_x \times d_{x'}$ submatrix $\xi (
  y, y')_{y \in S_x, y' \in S_{x'}}$ whose lines sum to $\frac{1}{d_x}$ and
  columns to $\frac{1}{d_{x'}}$. For instance, $\xi$ could be the uniform
  coupling
  \[ \xi = \frac{1}{d_x d_{x'}}  \left(\begin{array}{cccc}
       1 & \cdots & \cdots & 1\\
       \vdots &  &  & \vdots\\
       1 & \cdots & \cdots & 1
     \end{array}\right) \]
  where we have written only the submatrix.
  
  \item A variant from the above measure is the measure uniform on the ball
  $B_x = \{ x \} \cup S_x$.
\end{enumerate}
Ollivier's {\emph{coarse Ricci curvature}}{\footnote{Also called Wasserstein
curvature.}} between $x$ and $x'$ (which by the way need not be neighbors)
measures the ratio between the Wasserstein distance and the distance.
Precisely, we set
\[ \kappa^1 ( x, x') = 1 - \frac{W_1 ( \mu_x^1, \mu_{x'}^1)}{d ( x, x')}
   \cdot \]
Since $\mu_x^1$ is the uniform measure on the sphere, then $\kappa^1$ compares
the average distance between the spheres $S_x$ and $S_{x'}$ with the distance
between their centers, which indeed depends on the Ricci curvature in the smooth case 
(see \cite{Ollivier:2009p2415,Ollivier:2008p2417} for the analogy with riemannian manifolds 
which prompted this definition).

We will rather use the definition of Lin, Lu \& Yau \cite{Lin:2011tt} 
(see also Ollivier \cite{Ollivier:2008p2417}) for a smooth time variable~$t$: 
let $\mu^t_x$ be the {\emph{lazy random walk}}
\[ 
\mu_x^t ( y) = \left\{ \begin{array}{cl}
     1 - t & \text{if } y = x\\
     \frac{t}{d_x} & \text{if } y \in S_x\\
     0 & \text{otherwise}
   \end{array} \right. 
\]
so that $\mu_x^t = ( 1 - t) \delta_x + t \mu_x^1$ interpolates linearly
between the Dirac measure and the uniform measure on the sphere{\footnote{In
\cite{Lin:2011tt}, a different notation is used: the lazy random walk is parametrized by
$\alpha = 1 - t$ and the limit point corresponds to $\alpha = 1$.}}. We let
$\kappa^t ( x, x') = 1 - \frac{W_1 ( \mu^t_x, \mu^t_{x'})}{d ( x, x')}$, and
$\kappa^t ( x, x') \rightarrow 0$ as $t \rightarrow 0$. We then set
\[ 
\ric ( x, x') = \liminf_{t \rightarrow 0}  \frac{\kappa^t ( x, x')}{t} 
\]
and we will call $\ric$ the (asymptotic) {\emph{Ollivier--Ricci curvature}}. The
curvature $\ric$ is attached to a continuous Markov process whereas
$\kappa^1$ corresponds to a time-discrete process{\footnote{However both
approaches are equivalent, by considering weighted graphs and allowing loops
(i.e. weights $w_{x x}$). See \cite{Bauer:2011uu} and also \S\ref{sec:varying-length} 
for weighted graphs.}}. 
Lin, Lu \& Yau \cite{Lin:2011tt} prove the existence of the limit $\ric ( x, x')$ using
concavity properties. In the next section, we give a different proof by
linking the existence to a linear programming problem with convexity
properties.

The relevance of such a definition comes from the analogy with riemannian
manifolds but can also be seen through its applications, e.g. the existence of
an upper bound for the diameter of $X$ depending on $\ric$ (see Myers' theorem below).


\section{A linear programming problem}

In the case of graphs, the computation of $W_1$ is surprisingly simple to
understand and implement numerically. Recall that a coupling $\xi^t$ between
$\mu_x^t$ and $\mu_{x'}^t$ is completely determined by a $( d_x + 1) \times (
d_{x'} + 1)$ submatrix, and henceforward we will identify $\xi^t$ with this
submatrix. A coupling is actually any matrix in $\mathbb{R}^{( d_x + 1)  (
d_{x'} + 1)}$ with nonnegative coefficients, subject to the following
$t$-dependent linear constraints: $\forall y \in B_x, \; \langle \xi^t, L_y
\rangle = \mu^t_x ( y)$ and $\forall y' \in B_{x'}, \; \langle \xi^t, C_{y'}
\rangle = \mu_{x'} ( y')$, for all $y \in B_x$ and $y' \in B_{x'}$, where
$L_y$ and $C_{y'}$ are the following matrices
\[ L_y = \left(\begin{array}{cccc}
     \cdots & 0 & \cdots & \cdots\\
     1 & \cdots & \cdots & 1\\
     \cdots & 0 & \cdots & \cdots\\
     \cdots & 0 & \cdots & \cdots
   \end{array}\right), \qquad 
   C_{y'} = \left(\begin{array}{cccc}
     \vdots & 1 & \vdots & \vdots\\
     0 & \vdots & 0 & 0\\
     \vdots & \vdots & \vdots & \vdots\\
     \vdots & 1 & \vdots & \vdots
   \end{array}\right) 
\]
$\langle M, N \rangle = \text{}^t MN$ is the standard inner product between
matrices. We will write the nonnegativity constraint $\langle E_{y y'}, \xi^t
\rangle \ge 0$, where $E_{y y'}$ is the basis matrix whose coefficients
all vanish except at $( y, y')$.
The set of possible couplings is therefore a {\emph{bounded}} {\emph{convex
polyhedron}} $K^t$ contained in the unit cube $[ 0, 1]^{(d_x + 1) (d_{x'} + 1)}$. 
In the following, we will also need the limit set $K^0
= \{ E_{x x'} \}$ which contains a unique coupling (see case 1 above).
\\

In order to compute $\kappa^t$, we want to minimize the cost function $c$,
which is actually linear:
\[ 
c(\xi^t) = \sum_{y, y' \in V} \xi^t ( x, y) d ( x, y) = \sum_{y \in B_x, y'
   \in B_{x'}} \xi^t ( x, y) d ( x, y) = \langle \xi^t, D \rangle 
\]
where $D$ stands for the distance matrix restricted to $B_x \times B_{x'}$, so
that $D$ is the (constant) $L^2$ gradient of~$c$. Clearly the infimum is reached,
and minimizers lie on the boundary of $K^t$. Then either the gradient~$D$ is
perpendicular to some facet of $K^t$, and the minimizer can be freely chosen
on that facet, or not, and the minimizer is unique and lies on a vertex of
$K^t$. Moreover the Kuhn--Tucker theorem gives a characterization of
minimizers in terms of Lagrange multipliers (a.k.a. Kuhn--Tucker vectors) :
$\xi^t$~minimizes $c$ on $K^t$ if and only if there exists $( \lambda_y)_{y \in
B_x}$, $( \lambda'_{y'})_{y' \in B_{x'}}$ and $( \nu_{y y'})_{y \in B_x, y'
\in B_{x'}}$ such that
\begin{equation}
  \nabla c = D = \sum_y \lambda_y L_y + \sum_{y'} \lambda'_{y'} C_y +
  \sum_{y', y'} \nu_{y y'} E_{y y'} \hspace{1em} \text{with \ } \nu_{y y'}
  \ge 0 \label{eq:Kuhn-Tucker}
\end{equation}
and
\begin{equation}
\forall y, y', \quad 
\nu_{y y'}  \langle E_{y y'}, \xi^t \rangle = \nu_{y y'} \, \xi^t ( y, y') = 0 
\label{eq:Kuhn-Tucker0}
\end{equation}
meaning that the Lagrange multipliers $\nu_{y y'}$ have to vanish unless the
inequality constraint is {\emph{active}} (or {\emph{saturated}}): $\xi^t ( y,
y') = 0$. As a consequence, (i) finding a minimizer is practically easy thanks
to numerous linear programming algorithms and (ii) proving rigorously that a
given $\xi^t$ is a minimizer requires only writing the relations
\eqref{eq:Kuhn-Tucker} and \eqref{eq:Kuhn-Tucker0} for $\xi^t \in K^t$.

The non-uniqueness is quite specific to the $W_1$ metric, when cost is
proportional to length and therefore linear instead of strictly convex 
(instead of, say, length squared as in the 2-Wasserstein metric $W_2$). 
It corresponds to the following geometric fact: transporting mass $m$ from $x$ to
$z$ is equivalent in cost to transporting the same mass $m$ from $x$ to $y$
and from $y$ to $z$, as long as $y$ is on a geodesic from $x$ to $z$ (and $\mu
(y) \ge m$, since we prohibit negative mass).
\\

Computing the Olivier--Ricci curvature requires a priori taking a derivative,
however it is actually much simpler due to the following lemma, which also
proves its existence, without the need for subtler considerations like in
\cite{Lin:2011tt}:

\begin{Lemma}
  \label{lemma-homothetic}For $t, s$ small enough, convex sets $K^t$
  and $K^{s}$ are homothetic. More precisely,
  \[ K^{s} - E_{x x'} = \frac{s}{t}  ( K^t - E_{x x'}) . \]
\end{Lemma}

\begin{proof}
  First write the constraint corresponding to the lazy random walk as $\mu_x^t
  ( y) = \delta_x ( y) + t \Delta_{x y}$, where $\Delta_{x x} = - 1$ and
  $\Delta_{x y} = \frac{1}{d_x}$ iff $y \sim x$. Let $\xi^t$ lie in $K^t$ and
  $t, s$ be positive. Then $\xi^s = \frac{s}{t} \xi^t +
  \left( 1 - \frac{s}{t} \right) E_{x x'}$ lies in $K^{s}$.
  Indeed
  \[ \langle L_y, \xi^s \rangle = \frac{s}{t}  ( \delta_x ( y) +
     t \Delta_{x y}) + \left( 1 - \frac{s}{t} \right) \delta_x ( y) =
     \delta_x ( y) + s \Delta_{x y} = \mu_x^{s} ( y) \]
  \[ \langle C_{y'}, \xi^s \rangle = \frac{s}{t}  ( \delta_{x'}
     ( y') + t \Delta_{x' y'}) + \left( 1 - \frac{s}{t} \right)
     \delta_{x'} ( y') = \delta_{x'} ( y') + s \Delta_{x' y'} =
     \mu_{x'}^{s} ( y') . \]
  We see immediately that $\xi^s_{y y'} = \frac{s}{t} \xi^t_{y y'}
  \ge 0$ whenever $( y, y') \neq ( x, x')$. Moreover $\xi^s_{x x'}
  = \frac{s}{t} \xi^t_{x x'} + 1 - \frac{s}{t} \ge 0$,
  provided $\frac{s}{t} \le 1$. All the previous arguments hold
  in generality, but the nonnegativity of $\xi^s_{x x'}$ needs a
  different argument when $s \ge t$. Because $\mu_x^t$ and
  $\mu_{x'}^t$ are probability measures, $\sum_y \mu^t_x ( y) = \sum_{y'}
  \mu_{x'}^t ( y') = 1$, and for any $t$,
  \begin{eqnarray*}
    0 \le \sum_{y \neq x} \left( \sum_{y' \neq x'} \xi^t_{y y'} \right) &
    = & \sum_{y \neq x} ( \mu_x^t ( y) - \xi^t_{y x'}) = ( 1 - \mu_x^t ( x)) -
    \sum_{y \neq x} \xi^t_{y x'}\\
    & = & 1 - \mu_x^t ( x) - \big( \mu_{x'}^t ( x') - \xi^t_{x x'} \big)
  \end{eqnarray*}
  consequently $\xi^t_{x x'} \ge \mu_x^t ( x) + \mu_{x'}^t ( x') - 1 = 1 +
  t ( \Delta_{x x} + \Delta_{y y}) = 1 - 2 t$. Hence, for $t < \frac{1}{2}$, 
  $\xi^t_{x x'}$ is positive for any matrix $\xi^t$ satisfying the equality constraints 
  (and the same holds for $\xi^s$, using again that $t < \frac{1}{2}$).
  
  The signification of this positivity is that the constraint $\xi^t_{x x'}
  \ge 0$ is never saturated: there will always be some mass transported
  from $x$ to $x'$ if $t$ is small enough, because the other vertices cannot
  hold all the mass from $x$.
\end{proof}

\begin{Remark}
  The lemma holds true for $t, t'$ small enough, as long as $\Delta_{x x}$ is
  uniformly bounded on $X$, a~property which we will meet later. More precisely
  if $| \Delta_{x x'} | \le C$, then the homothety property holds for
  all $t \le \frac{1}{2 C}$. 
\end{Remark}

\begin{Proposition}
  The Ollivier--Ricci curvature $\ric$ is equal to any quotient
  $\kappa^t / t$ for $t$ small enough (e.g. $t \le \sfrac{1}{2}$).
\end{Proposition}

\begin{proof}
  As a consequence of lemma \ref{lemma-homothetic}, the gradient $D$ has the
  same projection on the affine space determined by the equality constraints,
  and the minimizers can be chose to be homothetic for $t \le t_0$ small
  enough. If $( \xi^t)$ denotes this family of homothetic minimizers:
  \[ 
  W_1 ( \mu_x^t, \mu_{x'}^t) = \langle D, \xi^t \rangle = \langle D,
     \frac{t}{t_0} \xi^{t_0} + \left( 1 - \frac{t}{t_0} \right) E_{x x'}
     \rangle = \frac{t}{t_0} W_1 ( \mu_x^{t_0}, \mu_{x'}^{t_0}) + \left( 1 -
     \frac{t}{t_0} \right) d ( x, x') 
  \]
  \[ \kappa^t ( x, x') = 1 - \frac{W_1 ( \mu_x^t, \mu_{x'}^t)}{d ( x, x')} = 1
     - \frac{t}{t_0}  \frac{W_1 ( \mu_x^{t_0}, \mu_{x'}^{t_0})}{d ( x, x')} -
     \left( 1 - \frac{t}{t_0} \right) = \frac{t}{t_0}  \left( 1 - \frac{W_1 (
     \mu_x^{t_0}, \mu_{x'}^{t_0})}{d ( x, x')} \right) = \frac{t}{t_0}
     \kappa^{t_0} 
  \]
  So $\kappa^t$ is linear for $t$ small enough and
  \[ 
  \ric ( x, x') = \frac{\mathd \kappa^t ( x, x')}{\mathd t}_{| t = 0
     } = \frac{\kappa^{t_0} ( x, x')}{t_0} \qedhere
  \]  
\end{proof}

As a consequence, computing $\ric ( x, x')$ is quite simple: one needs
only solve the linear problem for $t$ small enough (e.g. $t \le \sfrac{1}{4} $).

\begin{Remark}
  This property linking the time-continuous Olivier--Ricci curvature to the
  time-discrete curvature is true in generality, as soon as the random walk is
  lazy enough, i.e. the probability of staying at $x$ is large enough  
  (see \textup{\cite{Veysseire:2012wa}}).
\end{Remark}

Finally, we note that this optimization problem is an instance of \emph{integer} 
linear programming and as a consequence, the solution is integer-valued up to
a multiplicative constant:
\begin{Theorem}
	For any pair of adjacent vertices $x,x'$ with degrees $d,d'$, and $t=1/N$,
	there exists an optimal coupling $\xi^t$ with coefficients in $\frac{1}{N d d'} \mathbb{N}$; 
	consequently $\kappa^1(x,x')$ and $\ric(x,x')$ lie in $\frac{1}{d d'} \mathbb{Z}$.
\end{Theorem}
\begin{proof}
Let us first rewrite the constraints above as the following single linear equation.
Numbering $x=x_0$ and $(x_1,\dots,x_d)$ the neighbours of $x$, $x'=x'_0$ and $(x'_1,\dots,x'_d)$ 
the neighbours of $x'$, we consider the vector 
\[
X=(\xi^t(x_0,x'_0),\ldots,\xi^t(x_0,x'_{d'}),
\xi^t(x_1,x'_0),\ldots,\xi^t(x_1,x'_{d'}),\ldots, \xi^t(x_d,x'_0),\ldots,\xi^t(x_d,x'_{d'})).
\]
The constraints amounts to $A X = b$ for the following data
\[
\begin{pmatrix}
		1&		\cdots&	1&		&		&		&		&		&		\\
		&		&		&		&		\cdots&	&		&		&		\\
		&		&		&		&		&		&		1&		\cdots&	1\\
		1&		&		&		&		&		&		1&		&		\\
		&		\ddots&	&		&		\cdots&	&		&		\ddots&	\\
		&		&		1&		&		&		&		&		&		1
	\end{pmatrix} 
	\begin{pmatrix}
		\xi^t(x_0,x'_0)\\
		\vdots\\
		\xi^t(x_0,x_{d'})\\
		\xi^t(x_1,x'_0)\\
		\vdots\\
		\xi^t(x_d,x'_0)\\
		\vdots\\
		\xi^t(x_d,x_{d'})
	\end{pmatrix}
	=
	\begin{pmatrix}
		\mu^t_{x_0}(x_0)\\
		\vdots\\
		\mu^t_{x_0}(x_d)\\
		\mu^t_{x'_0}(x'_0)\\
		\vdots\\
		\mu^t_{x'_0}(x'_{d'})
	\end{pmatrix}
\]
The integral matrix $A$ is \emph{totally unimodular}: Every square, non-singular submatrix $B$ of $A$
has determinant $\pm 1$. Indeed $A$ satisfies the following requirements:
\begin{itemize}
	\item the entries of $A$ lie in $\{-1,0,1\}$
	\item $A$ has no more than 2 nonzero entries on each column
	\item its rows can be partitioned into two sets $I_1 = \{1,\dots,d+2\}$ and 
	$I_2 = \{d+3,\dots,d+d'+2\}$ such that if a column has two entries of the same sign, 
	their rows are in different sets.
\end{itemize}
Then, whenever $b$ is integer-valued, the vertices of the constraint set 
$\{ X \in \mathbb{R}_+^{(d+1)(d'+1)},\;AX=b\}$ are also integer-valued. 
We refer the reader to classical results of integer linear programming which can be found 
in~\cite{Papadimitriou:1998tv}.

In our setting, choose $t=1/N$, so that the coefficients of $b$ lie in 
$\frac{1}{N d d'} \mathbb{N}$. By the above remarks, so do the coefficients of $\xi^t$, 
since an optimal coupling can be chosen to be a vertex of the contraint set. 
Since the distance matrix in also integer-valued, the cost $W_1$ lies 
in $\frac{1}{N d d'} \mathbb{N}$, and for two neighbors $x,x'$, 
$\kappa^t(x,x') \in \frac{1}{N d d'} \mathbb{N}$.
The curvature $\ric(x,x')$ is obtained by diving by $t=1/N$, hence the result.
The reasoning also holds for $\kappa^1$.
\end{proof}


\section{Curvature of discrete surfaces}

Estimates for the Ollivier--Ricci curvature are given in \cite{Lin:2011tt} and \cite{Jost:2011wt}
($\ric$ in the first paper, $\kappa^1$ in the second) for general graphs
and for some specific ones such as trees. Essentially they rely on studying
one coupling, which gives an upper bound on $W_1$, hence a lower bound on the
curvature, which may or may not be optimal. We will give below exact values,
albeit in the specific setting which concerns us: polyhedral surfaces.
Furthermore, we will always assume that vertices $x, x'$ are neighbors;
in other words, we see $\ric$ as a function on the edges. 
Actual computing of $\ric ( x, x')$ for more distant
vertices is of course possible, but much more complicated. However, it should
be noted that $\ric$ trivially enjoys a concavity property, as a direct
consequence of the triangle inequality on the distance $W_1$: if $x = x_0,
x_1, \ldots, x_n = x'$ is a geodesic path from $x$ to $x'$ then
\begin{equation}
  \kappa^t ( x_0, x_n) \ge \sum_{i = 1}^n \frac{d ( x_{i - 1}, x_i)}{d (
  x_0, x_n)} \kappa^t ( x_{i - 1}, x_i) = \frac{1}{d ( x_0, x_n)}  \sum_{i =
  1}^n \kappa^t ( x_{i - 1}, x_i)  \label{eq:min-geodesic}
\end{equation}
the latter equality holding only in the uniform metric, because $d ( x_{i -
1}, x_i) = 1$ between neighbors. This inequality passes to the limit and
applies to $\ric$ as well. The concavity property implies in particular
that if $\ric$ is bounded below on all edges, then $\ric ( x, y)$
has the same lower bound on all couples $x, y$.

We use this fact to give a trivial proof of Myers' theorem 
(see also \cite{Gallot:1990gm} for the smooth case).

\begin{Theorem}[Ollivier {\cite[prop.~23]{Ollivier:2009p2415}}]  \label{thm:BM}
  If $\ric$ is bounded below on all edges by a positive
  constant $\rho$, then $S$ is finite, and its diameter is bounded above by $2/\rho$.
\end{Theorem}

\begin{proof}
  Using the triangle inequality again on $W_1$
  \begin{eqnarray*}
    d ( x, y) & = & W_1 ( \delta_x, \delta_y)\\
    & \le & W_1 ( \delta_x, \mu_x^t) + W_1 ( \mu_x^t, \mu_y^t) + W_1 (
    \mu_y^t, \delta_y)\\
    & \le & J^t ( x) + ( 1 - \kappa^t ( x, y)) d ( x, y) + J^t ( y)
  \end{eqnarray*}
  where $J^t ( x) = W_1 ( \delta_x, \mu_x^t)$ is the {\emph{jump}} at $x$, 
  which is also the expectation $\mathbb{E}_{\mu_x^t}(d(x,.))$ of the distance to~$x$ w.r.t. 
  the probability $\mu_x^t$. For the uniform metric $J^t ( x) = t$, so that
  \[ d ( x, y) \le \frac{2}{\ric ( x, y)} \le \frac{2}{\rho}
  \]
  which gives the upper bound for the diameter. Since $S$ is locally finite,
  it is therefore finite.
\end{proof}
\bigskip

We will now give our results, and compare them with those obtained either by Jost \& Liu
\cite{Jost:2011wt} or by using Forman's definitions of Ricci curvature \cite{Forman:2003p1300}.

As first example, let us give the Ollivier--Ricci curvature for the Platonic
solids (with $\kappa^1$ as a comparison, corresponding to the non-lazy
random walk) in table~\ref{table:platonic}:
\begin{table}[h]	\centering
	\setlength{\mycolumnwidth}{22mm}
	\begin{tabular}{|c|*{5}{x{\mycolumnwidth}|}}
	\hline
	& tetrahedron & cube & octahedron & dodecahedron & icosahedron \\  \hline
	$\ric$ & $4/3$ & $\mathbf{2/3}$ & $\mathbf{1}$ & $0$ & $2/5$ \\ 	\hline
   $\kappa^1$ & $2/3$ & $0$ & $1/2$ & $- 1/3$ & $1/5$ \\   \hline
   $\frac{1}{3}$ Forman & $4/3$ & $\mathbf{2/3}$ & $2 / 3$ & $0$ & $0$ \\  \hline
	\end{tabular}
\caption{Ollivier-Ricci (asymptotic and discrete at time $1$ for the Platonic solids, 
along Forman's version of Ricci curvature (divided by $3$, to be comparable).}
\label{table:platonic}
\end{table}
This stresses the difference between~$\kappa^1$ (used in \cite{Jost:2011wt}) and
$\ric$, which exhibits, in our opinion, a more geometric{\footnote{And
less graph-theoretic.}} behavior. {\emph{In~particular, the values of
$\ric$ are sharp w.r.t. Myers' theorem}} for the cube and the octahedron. 
Forman refers to the combinatorial Ricci curvature 
for unit weights defined in~\cite{Forman:2003p1300}, 
which also satisfies a Myers' theorem, albeit with a different constant: the
diameter is bounded above by $6/\rho$, hence our choice to divide it by $3$,
to allow comparison between with the Ollivier--Ricci curvature. Here, only the
cube is optimal.

Tessellations by regular polygons fit well in this framework since all edges
have the same length. Regular tiling are the triangular, square and hexagonal
tiling. The triangular tiling corresponds to the $( 6, 6)$ case above and has
zero Ollivier--Ricci curvature, and so does the square tiling. However the
hexagonal lattice has negative Ollivier--Ricci curvature equal to $-2/3$.

The method can also be applied to semiregular tiling, but those are only vertex-transitive in general
and not edge-transitive (with the exception of the trihexagonal tiling), 
hence one must treat separately the different types of edges. For example, for the snub square tiling, 
$\ric=0$ for an edge between two triangles, but $\ric=-1/5$ for an edge between 
a triangle and a square.
\\

The results above can easily be derived using making computation by hand or by
using integer linear programming software (a program with all the above
examples using opensource software \texttt{Sage}\footnote{\url{http://www.sagemath.org/}} 
is attached to the article).
The next results however are of a more general nature, with variable degrees, and 
cannot be obtained by simple computations. We consider adjacent vertices $x, x'$ on a
triangulated surface with the following genericity hypotheses:
\begin{description}
  \item[(B)] $x, x'$ are not on the boundary,
  
  \item[(G)] for any $y \in \text{star} ( x)$ and $y' \in \text{star} ( x')$,
  there is a geodesic of length $d (y, y')$ in $\text{star} ( x) \cup
  \text{star} ( x')$.
\end{description}
Under hypothesis (G), the distance matrix $D$ in $X$ agrees with its
restriction to $\text{star} ( x) \cup \text{star} ( x')$, hence all
computations are local. For this genericity assumption to fail, one needs very
small loops close to $x$ and $x'$, which can usually be excluded as soon as
the triangulation is fine enough. Note that the Platonic solids are not
generic in that sense, and many other configurations are ruled out (e.g. $( 3,
4)$). Then we conclude with the following.

\begin{Theorem}
  Under the genericity hypotheses (B) and (G), the Ollivier--Ricci curvature
  depends only on the degrees $d, d'$ of vertices $x, x'$ and is given in
  table~\ref{table:degrees}.
\end{Theorem}

\begin{proof}
  To compute the optimal cost $W_1 ( \mu_x^t, \mu_{x'}^t)$, we need only find
  a coupling $\xi^t$ for which the Kuhn-Tucker relation (\ref{eq:Kuhn-Tucker})
  holds. Thanks to the genericity hypothesis (G), we can restrict ourselves to
  finite matrices (on $\text{star} ( x) \cup \text{star} ( x')$). Details are
  given in the section~\ref{appendix}.
  Note that the hypothesis (B) makes for simpler calculations, but they could
  obviously be extended to deal with the presence of boundary. 
\end{proof}
\begin{table}[h]
	\[ 
	\setlength{\mycolumnwidth}{11mm}
	\renewcommand*{\arraystretch}{1.8} 
	\begin{array}{|c|*{6}{y{\mycolumnwidth}|}c|}
     \hline
     ( d, d') & ( 3, 3) & ( 4, 4) & ( 4, 5) & ( 4, 6) & ( 5, 5) & ( 5, 6) &
     \text{others}\\
     \hline
     \ric & \frac{4}{3} & \frac{3}{4} & \frac{11}{20} & \frac{1}{3} & \frac{2}{5} & \frac{2}{15} &
     \displaystyle \frac{4}{d} + \frac{8}{d'} - 2 \\[5pt]
     \hline
     \kappa^1 & \frac{2}{3} & \frac{1}{2} & \frac{7}{20} & \frac{1}{4} & \frac{1}{5} & \frac{1}{10} 
     & \displaystyle \frac{4}{d} + \frac{8}{d'} - 2 \\[5pt]
     \hline
     \kappa^1 \ge & \frac{2}{3} & \frac{1}{2} & \frac{1}{5} & 0 & 0 & - \frac{1}{5} & \left\{
     \begin{array}{ll}
       \displaystyle \frac{5}{d'} - \frac{2}{3} & \text{ if } d = 3 \\[5pt]
       \displaystyle \frac{4}{d} + \frac{6}{d'} - 2 & \text{ if } d \ge 4
     \end{array} \right. \\[25pt]
     \hline
     \displaystyle \frac{1}{3}\, \text{Forman} & \frac{4}{3} & \frac{2}{3} & \frac{1}{3} & 0 & 0 
     & - \frac{1}{3} & \displaystyle  \frac{10 - d - d'}{3} \\[5pt]
     \hline
   \end{array} 
   \]
   \caption{Asymptotic Ollivier--Ricci curvature $\ric(x,x')$ according to respective 
   degrees of $x$ and $x'$, compared to the time $1$ Ollivier--Ricci $\kappa^1$, as well as 
   Forman's Ricci curvature (divided by $3$ for comparison purposes); $\kappa^1 \ge$
   refers to the estimates of Jost \& Liu.}
   \label{table:degrees}
\end{table}
The table~\ref{table:degrees} gives $\ric$ and $\kappa^1$ in function of respective
degrees $d, d'$. Because $\ric ( x, x') = \ric ( x', x)$ we may
assume without loss of generality that $d \le d'$.
We compare with Forman's expression and also to the lower bound
\[ 
\frac{\sharp ( x, x')}{d'} - \left( 1 - \frac{1}{d} - \frac{1}{d'} -
   \frac{\sharp ( x, x')}{d} \right)_+ - \left( 1 - \frac{1}{d} - \frac{1}{d'}
   - \frac{\sharp ( x, x')}{d'} \right)_+ 
\]
given by Jost \& Liu \cite{Jost:2011wt} for general graphs, where $\sharp ( x, x')$ is the number of
triangles incident to the edge ($x x'$), which under our hypotheses is always
equal to $2$. Jost \& Liu conclude that the presence of triangles improves the
lower Ricci bound. We see here that when there are {\emph{only}} triangles one
obtains an actual value, which differs from their lower bound as soon as $d'
\ge 5$.

\begin{Remark}
\begin{enumerate}
  \item The $( 3, 3)$ case is given here although it contradicts either (B) or
  (G), the latter being the tetrahedron computed above; similarly the $( 3,
  4)$ case is excluded.
  \item Zero Ollivier--Ricci curvature is attained only
  with degrees $( 6, 6)$ (regular triangular tiling), $( 4, 8)$ and $( 3, 12)$. 
  \end{enumerate}
\end{Remark}


\section{Varying edge lengths}	\label{sec:varying-length}

While many authors have focused on the graph theory, the case of polyhedral
surfaces is somewhat different: The combinatorial structure is more
restrictive, as we have seen above, but the geometry is more varied. In
particular, edge lengths $d(x, y)$ may be different from one. 
This is partially achieved in the literature \cite{Lin:2011tt,Jost:2011wt}
by allowing {\emph{weights}} on the edges, which amounts to changing the random walk,
but we think the geometry should intervene at two levels:  measure and distance.
We will present here a general framework to approach the problem, using the Laplace operator, 
which depends on both the geometric and the combinatorial structure of $S$.
One must also note the ambiguous definition of the Ollivier--Ricci asymptotic curvature,
which plays the role of a length in Myers' theorem, and yet its definition
makes it a dimensionless quantity. Indeed multiplying all lengths by a constant $\lambda$ 
will not change $\ric$ ($W_1$ being multiplied by $c$ as well).
\\

In the following we assume that $S$ is a polyhedral (or discrete)
surface with set of vertices $V$, edges $E$ and faces $F$. Furthermore $S$
is not only locally finite, but its vertices have a maximum degree
$d_{\max}$ ($d_{\min}$ denotes the minimal degree, which is at least $2$ for
surfaces with boundary, and $3$ for surfaces without boundary). The geometry
of $S$ is determined by the geometry of its faces, namely a isometric
bijection between each face $f$ and a planar face of identical degree, with
the compatibility condition that edge lengths measured in two adjacent faces
coincide. Then two natural notions of length arise: (i) the combinatorial
length, which counts the number of edges along a path and (ii) the metric
length, where each edge length is given by the geometry. Each notion of length
yields a different distance between vertices: the combinatorial distance
$\bar{d}$, which we have used above, and the metric distance $d$. Note that if
each face is assumed to be a regular polygon with edges of length one, then
both distances agree, and metric theory coincides with graph theory. We will
make the following assumption on the geometry: the distance $d$ and $\bar{d}$
are metrically equivalent: $\exists C, \; C^{- 1}  \bar{d} \le d
\le C \bar{d}$. Such an hypothesis holds if the lengths of edges are
uniformly bounded above and below; in particular, the aspect ratio is
bounded{\footnote{This also rules out extremely large or extremely
small faces, which could happen with only the bounded aspect ratio.}}.
\\

We consider a differential operator $\Delta$ (a laplacian, see \cite{ColindeVerdiere:1998vb}) 
determined by its values $\Delta_{x y}$ for vertices $x,y$ and the usual
properties{\footnote{Note that our sign convention is such that the Laplacian
is a negative operator; \cite{ColindeVerdiere:1998vb} uses the opposite.}}:
\begin{enumerate}[(a)]
  \item $\Delta_{x y} > 0$ whenever $x \sim y$,
  
  \item $\Delta_{x y} = 0$ whenever $x \neq y$ and $x \nsim y$ (locality
  property)
  
  \item $\sum_y \Delta_{x y} = 0$, which implies that $\Delta_{x x} < 0$ 
  (note that the sum is finite due the previous assumption and the local finiteness
  of $S$).
\end{enumerate}
Often this operator is obtained by putting a weight $w_{x y}=w_{y x}$ on each edge $(x y)$. 
The degree at $x$ is then the sum $d_x = \sum_{y \sim x} w_{x y}$ and 
$\Delta_{x y}= w_{x y}/d_x$. Obviously property (c) implies $\Delta_{x x} = -1$. 
The case studied above corresponds to a graph with all weights equal to one (therefore 
unweighted), and the corresponding Laplace operator is called the {\emph{harmonic
laplacian}} $\bar{\Delta}$.

The laplacian is not a priori symmetric, i.e. $L^2$-self-adjoint (though it
could be made so w.r.t. some metric on vertices). Thanks to the finiteness assumption (b), 
we can define iterates $\Delta^k$ of $\Delta$ for integer $k$, and the $x y$
coefficient (not to be confused with $( \Delta_{x y})^k$) is
\[ 
\Delta^k_{x y} = \sum_{z_1, \ldots, z_{k - 1}} \Delta_{x z_1} \Delta_{z_1 z_2} 
\cdots \Delta_{z_{k - 1} y} 
\]
the sum being taken on all paths of length $k$ on $S$. By direct recurrence,
we see that our boundedness hypotheses imply the bound $| \Delta_{x y}^k | \le 2^k$. 
Indeed,
\[ 
\left| \sum_z \Delta_{x z} \Delta^k_{z y} \right| \le 2^k  \left|
   \sum_z \Delta_{x z} \right| \le 2^k  \left( 1 + \sum_{z \neq x} |
   \Delta_{x z} | \right) = 2^k  \left( 1 - \sum_{z \neq x} \Delta_{x z}
   \right) = 2^{k + 1} . 
\]
As a consequence, the heat semigroup $e^{t \Delta} = \sum_{k = 0}^{\infty} 
\frac{t^k}{k!} \Delta^k$ is well-defined. It acts on measures, and defines the
image measure $\delta_x^t = \delta_x e^{t \Delta}$ of the Dirac measure at $x$
by
\[ 
\delta_x^t (y) = \sum_z \delta_x (z)  ( e^{t \Delta})_{z y} 
= ( e^{t \Delta})_{x y} = \mu_x^t (y) +\mathcal{O} ( t^2) 
\]
where $\mu_x^t (y) = \delta_x (y) + t \Delta_{x y}$ is the lazy random walk
studied above (for the harmonic laplacian, but results hold in the general
case). The random walks $(\delta_x^t)_{x \in V}$ have finite first moment, as
can be inferred from the proof of the following.

\begin{Proposition}
  The Ollivier--Ricci curvature depends only on the first order expansion of
  the random walk:
  \[ \lim_{t \rightarrow 0} \frac{1}{t}  \left( 1 - \frac{W_1 ( \delta_x^t,
     \delta_y^t)}{d (x,y)} \right) = \lim_{t \rightarrow 0} \frac{1}{t} 
     \left( 1 - \frac{W_1 ( \mu_x^t, \mu_y^t)}{d ( i, j)} \right) . 
  \]
\end{Proposition}

\begin{proof}
  Consider any coupling $\xi$ that transfers mass from points at (uniform)
  distance $\bar{d}$ from $x$ at least 2, to $x$ and its neighbors. If the
  vertex $y$ is at $\bar{d}$-distance $k$ from $x$, then $\Delta^{\ell}_{x y}
  = 0$ for $\ell < k$ and
  \[ | \delta^t_x ( y) | = \left| \sum_{\ell \ge k} \frac{t^{\ell}}{\ell
     !} \Delta^{\ell}_{x y} \right| \le \sum_{\ell \ge k} \frac{(
     2 t)^{\ell}}{\ell !} \le \frac{( 2 t)^k}{k!} e^{2 t} . \]
  The points at uniform distance $k$ from $x$ are at most $d_{\max}^k$
  numerous, and using the equivalence between distances, they will be moved at
  most by $d \le C ( k + 1)$ to $x$ or one of its neighbors:
  \begin{eqnarray*}
    W_1 ( \delta_x^t, \mu_x^t) & \le & \sum_{k = 2}^{\infty} C ( k + 1)
    d_{\max}^k  \frac{( 2 t)^k}{k!} e^{2 t} \le \frac{3 Ce^{2 t}}{2}
    \sum_{k = 2}^{\infty}  \frac{( 2 td_{\max})^k}{( k - 1) !}\\
    & = & 3 Cd_{\max} te^{2 t}  \sum_{k = 1}^{\infty}  \frac{( 2
    td_{\max})^k}{k!} \le 3 Cd_{\max} t^2 e^{2 t} e^{2 t d_{\max}}
    =\mathcal{O} ( t^2) .
  \end{eqnarray*}
  Since
  \[ | W_1 ( \delta_x^t, \delta_y^t) - W_1 ( \mu_x^t, \mu_y^t) | \le
     W_1 ( \delta_x^t, \mu_x^t) + W_1 ( \delta_y^t, \mu_y^t) =\mathcal{O} (
     t^2) \]
  we conclude that both limits coincide.
\end{proof}

As a consequence, it is natural to replace in the section above the random walk by 
$\mu_x^t = \delta_x + t \Delta_{x,.}$, for some definition of the Laplacian 
(see \cite{Bobenko:2007fl,Wardetzky:2008vh,Alexa:2011td}). However in order to recover 
the geometric properties above one needs to normalize the random walk $\mu_x^t$,
so that the jump $J^t(x)=t$, i.e. the average distance of points 
jumping from $x$ should be $t$. That amounts to setting:
\[
\mu_x^t(y) = \delta_x(y) + \frac{t \, \Delta_{x y}}{\sum_{z \sim x} d(x,z) \Delta_{x z} }
\]
equivalently one might renormalize the laplacian accordingly. As a consequence, 
$\ric$ now behaves as the inverse of a length, as expected. Furthermore, 
Myers' theorem~\ref{thm:BM} is still valid. Indeed, while equation~\eqref{eq:min-geodesic}
no longer holds when edge lengths vary, it remains true that $\ric(x,y) \ge \rho$ if $\ric$ is 
bounded below on all edges by $\rho$.
\\

\emph{An example: The rectangular parallelepiped.}

For the rectangular parallelepiped with edges of lengths $a,b,c$, the Ollivier--Ricci curvature is
\[
\ric = \frac{1}{a} - \frac{1}{a + b +c}
\]
along an edge of length $a$, and others follow (see \S\ref{ann:parallelepiped}).
For the cube, we recover $\ric = \frac{2}{3 a}$. 
If $a$ is the length of the longest edge, an application of Myers' theorem 
yields an upper bound for the diameter $\frac{2 a}{b+c}$ times greater than its actual value $a+b+c$.

\begin{Remark}
  A more general theory can be developed with non-local operators, by replacing local finiteness 
  (property (b) above) with convergence requirements. Another, still finite, natural generalization 
  of (b) is to allow $\Delta_{x y} \neq 0$ whenever $x$ and $y$ belong to the same face. 
  For a triangulated
  manifold this amounts to the usual neighborhood relation, but as soon as some faces have more than
  three edges, this makes a difference (e.g. the cube). Note however that the corresponding 
  Myers' theorem needs to be adjusted as well since the jump will change accordingly. 
  In our experiments on Platonic solids with $\mu_x$ a uniform measure on vertices of $\text{star}(x)$, 
  we did not find better diameter bounds with this method.
\end{Remark}

\begin{Remark}
	One might also be tempted to compute Ollivier--Ricci curvature on the surface $S$ seen as a smooth 
	flat surface with conical singularities (so that distances are computed between points \emph{on} 
	the faces). If vertices $x,x'$ both have nonnegative Gaussian curvature
	(a.k.a. angular defect $\alpha,\alpha' \in \mathbb{R}_+$) then by a computation analog to Ollivier's
	\cite{Ollivier:2009p2415}, we infer
	\[
	\ric = \frac{4}{3}  \left( 
	\frac{1}{2 \pi - \alpha'} \sin \frac{\alpha'}{2} + \frac{1}{2 \pi - \alpha} \sin \frac{\alpha}{2} 
	\right)
	\]
	which differs from our previous computations. This emphasizes that this setup is somewhere
	in between the smooth and the discrete setup.
\end{Remark}


\section{Appendix: solutions for the linear programming problem on generic
triangulated surfaces} \label{appendix}

We give here the Lagrange multipliers for the linear programming problem and
the corresponding minimizer. The regular tetrahedron is given first as an
example of the method, and the main result consists of analyzing the various
cases according to their (arbitrary) degrees. Cases with degrees less or equal to $6$ 
can easily be computed by a machine and we refer to the \texttt{Sage} 
program attached.

\subsection{The regular tetrahedron}

The distance matrix for vertices $1, 2, 3, 4$ is
\[ D = \left(\begin{array}{cccc}
     0 & 1 & 1 & 1\\
     1 & 0 & 1 & 1\\
     1 & 1 & 0 & 1\\
     1 & 1 & 1 & 0
   \end{array}\right) \]
and the optimal coupling from $\mu_1^t$ to $\mu_2^t$ shifts mass $1 - \frac{4
t}{3}$ from vertex $1$ to vertex $2$ (provided $t \le $), leaving other
vertices untouched:
\[ \xi^t = \left(\begin{array}{cccc}
     \frac{t}{3} & 1 - \frac{4 t}{3} & 0 & 0\\
     0 & \frac{t}{3} & 0 & 0\\
     0 & 0 & \frac{t}{3} & 0\\
     0 & 0 & 0 & \frac{t}{3}
   \end{array}\right) \hspace{1em} \text{ with cost } \hspace{1em} \langle
   \xi^t, D \rangle = 1 - \frac{4 t}{3} \]
with Lagrange multipliers
\[ D = \left(\begin{array}{cccc}
     0 & 1 & 1 & 1\\
     1 & 0 & 1 & 1\\
     1 & 1 & 0 & 1\\
     1 & 1 & 1 & 0
   \end{array}\right) = C_2 - L_2 + \left(\begin{array}{cccc}
     0 & 0 & 1 & 1\\
     2 & 0 & 2 & 2\\
     1 & 0 & 0 & 1\\
     1 & 0 & 1 & 0
   \end{array}\right) \]
the last matrix corresponding to a linear combination of $E_{y y'}$ with
positive coefficients $\nu_{y y'}$, only where $\xi^t ( y, y') = 0$.
Conversely, it is straightforward from $\nu_{y y'} \xi^t ( y, y') = 0$ to
deduce that $\xi^t$ is unique. Hence $\kappa^t = \frac{4 t}{3}$ and
$\ric = 4/3$. The case $t = 1$ cannot be dealt with in the same way,
but admits the following optimal transference plan
\[ \xi^1 = \left(\begin{array}{cccc}
     0 & 0 & 0 & 0\\
     1/3 & 0 & 0 & 0\\
     0 & 0 & 1/3 & 0\\
     0 & 0 & 0 & 1/3
   \end{array}\right), \hspace{1em} D = L_2 - C_2 + \left( \begin{array}{cccc}
     0 & 2 & 1 & 1\\
     0 & 0 & 0 & 0\\
     1 & 2 & 0 & 1\\
     1 & 2 & 1 & 0
   \end{array} \right) \]
with cost $1/3$ and therefore curvature $\kappa^1 = 2/3$.

\begin{Remark}
  The case of degrees $( 3, 3)$ differs only in that the distance between
  vertices $3$ and $4$ is equal to $2$ instead of $1$. However the optimal
  couplings found above do not move mass from $3$ nor from $4$. Hence it is
  also optimal for the $( 3, 3)$ case. 
\end{Remark}


\subsection{Generic triangulated surfaces}

We analyse now generic triangulated surfaces according to the degrees $d \leq d'$ of
$x$ and $x'$. In our matrix notation, $x$ will have index $1$ and $x'$ index
2. Since $x$ and $x'$ are not on the boundary, all edges containing them
belong to two triangular faces. In particular there are two vertices, with
indices $3$ and $4$, that are neighbors of both $x$ and $x'$ (see figure~\ref{fig:generic-triangulation}).
There remains $d - 3$ {\emph{exclusive}} neighbors of $x$ (that are not
neighbors of $x'$), ordered from $5$ to $d + 1$ along the border of
$\text{star} ( x)$, and $d' - 3$ exclusive neighbors of $x'$, ordered from $d
+ 2$ to $n = d + d' - 2$ along the border of $\text{star} ( x')$.

\begin{figure}[h] 	\centering
\begin{tikzpicture}[scale=1.5]
		\node [style=vertex] (0) at (-0.5, 0) {$1$};
		\node [style=vertex] (1) at (2.5, 0) {$2$};
		\node [style=vertex] (2) at (1, 2) {$3$};
		\node [style=vertex] (3) at (1, -2) {$4$};
		\node [style=vertex] (4) at (-1.25, 2) {$5$};
		\node [style=vertex] (5) at (-2.5, 1.25) {$6$};
		\node [style=vertex] (6) at (-2.5, -1) {$d$};
		\node [style=vertex] (7) at (-1.25, -2) {$d+1$};
		\node [style=vertex] (8) at (2.5, -2) {$d+2$};
		\node [style=vertex] (9) at (3.75, -1.75) {$d+3$};
		\node [style=vertex] (10) at (4.3, -1) {$d+4$};
		\node [style=vertex] (11) at (3.75, 1.75) {$n$};
		\node [style=vertex] (12) at (2.5, 2) {$n-1$};
		\node [style=vertex] (13) at (4.3, 1) {$n-2$};

		\draw [style=outline] (0) to (1);
		\draw [style=outline] (0) to (3);
		\draw [style=outline] (3) to (1);
		\draw [style=outline] (1) to (2);
		\draw [style=outline] (2) to (0);
		\draw [style=outline] (0) to (4);
		\draw [style=outline] (0) to (5);
		\draw [style=outline] (0) to (6);
		\draw [style=outline] (0) to (7);
		\draw [style=outline] (1) to (8);
		\draw [style=outline] (1) to (9);
		\draw [style=outline] (1) to (10);
		\draw [style=outline] (1) to (11);
		\draw [style=outline] (1) to (12);
		\draw [style=outline] (2) to (4);
		\draw [style=outline] (4) to (5);
		\draw [thick,dashed] (5) to ++(-120:8mm);
		\draw [thick,dashed] (6) to ++(120:8mm);
		\draw [style=outline] (6) to (7);
		\draw [style=outline] (7) to (3);
		\draw [style=outline] (3) to (8);
		\draw [style=outline] (8) to (9);
		\draw [style=outline] (9) to (10);
		\draw [thick,dashed] (10) to ++(70:8mm);
		\draw [thick,dashed] (13) to ++(-70:8mm);
		\draw [style=outline] (11) to (12);
		\draw [style=outline] (12) to (2);
		\draw [style=outline] (13) to (11);
		\draw [style=outline] (1) to (13);
\end{tikzpicture}
\caption{Generic description of $\text{star} (1) \cup \text{star} (2)$.}
\label{fig:generic-triangulation}
\end{figure}
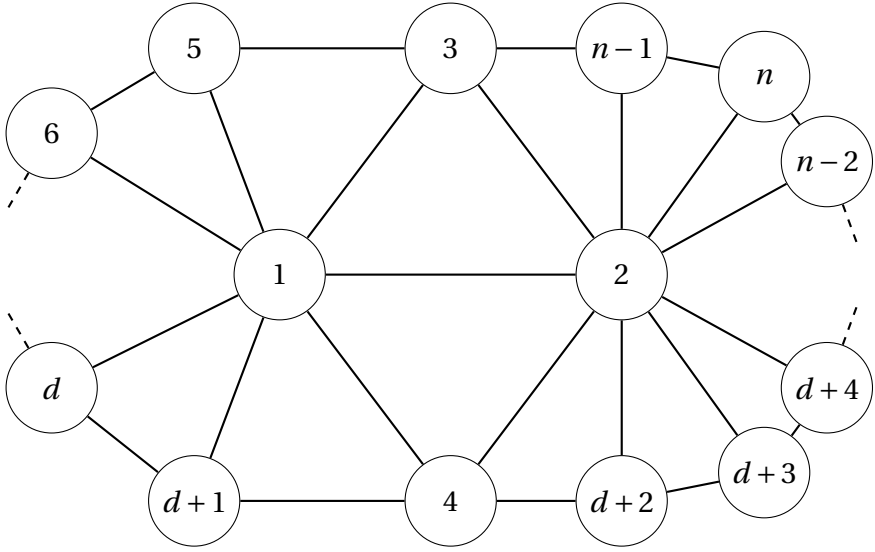

The distance matrix is
\[ 
D = \left(\begin{array}{cccc|ccccc|ccccc}
     0 & 1 & 1 & 1 & 1 & \cdots & \cdots & \cdots & 1 & 2 & \cdots & \cdots &
     \cdots & 2\\
     1 & 0 & 1 & 1 & 2 & \cdots & \cdots & \cdots & 2 & 1 & \cdots & \cdots &
     \cdots & 1\\
     1 & 1 & 0 & 2 & 1 & 2 & \cdots & \cdots & 2 & 2 & \cdots & \cdots & 2 &
     1\\
     1 & 1 & 2 & 0 & 2 & \cdots & \cdots & 2 & 1 & 1 & 2 & \cdots & \cdots &
     2\\
     \hline
     1 & 2 & 1 & 2 & 0 & 1 & 2 & \cdots & 2 & 3 & \cdots & \cdots & 3 & 2\\
     \vdots & \vdots & 2 & \vdots & 1 & \ddots & \ddots & \ddots & \vdots &
     \vdots &  &  &  & 3\\
     \vdots & \vdots & \vdots & \vdots & 2 & \ddots & \ddots & \ddots & 2 &
     \vdots &  &  &  & \vdots\\
     \vdots & \vdots & \vdots & 2 & \vdots & \ddots & \ddots & \ddots & 1 & 3
     &  &  &  & \vdots\\
     1 & 2 & 2 & 1 & 2 & \cdots & 2 & 1 & 0 & 2 & 3 & \cdots & \cdots & 3\\
     \hline
     2 & 1 & 2 & 1 & 3 & \cdots & \cdots & 3 & 2 & 0 & 1 & 2 & \cdots & 2\\
     \vdots & \vdots & \vdots & 2 & \vdots &  &  &  & 3 & 1 & \ddots & \ddots
     & \ddots & \vdots\\
     \vdots & \vdots & \vdots & \vdots & \vdots &  &  &  & \vdots & 2 & \ddots
     & \ddots & \ddots & 2\\
     \vdots & \vdots & 2 & \vdots & 3 &  &  &  & \vdots & \vdots & \ddots &
     \ddots & \ddots & 1\\
     2 & 1 & 1 & 2 & 2 & 3 & \cdots & \cdots & 3 & 2 & \cdots & 2 & 1 & 0
   \end{array}\right) 
\]
The form of the distance matrix is somewhat different when $d$ or $d'$ are very small. 
Indeed, due to the genericity assumption (B), both degrees are larger or equal to $4$. We see easily 
that the distance from an exclusive neighbor $y$ of $x$ to an exclusive neighbor $y'$ of $x'$
is in general (i.e. for $d-5$ neighbors $y$ and $d'-5$ neighbors $y'$) obtained 
by a geodesic passing though $x$ and $x'$. But when $d < 6$ or $d'< 6$, ``shortcuts'' predominate,
hence the need for \emph{ad hoc} computations.

Applying the constraints, we see that any coupling, and in particular the optimal coupling, takes 
the following block form
\[
\xi^t = \begin{pmatrix}
* & 0 & * \\ * & 0 & * \\ 0 & 0 & 0
\end{pmatrix}
\]
so, we may as well restrict to the lines $1$ through $d+1$ and columns $1,2,3,4$ and 
$d+2$ through $d+d'-2$ of matrices $D$ and $\xi^t$. Then, the $(d+1) \times (d'+1)$ submatrix 
$\tilde{D}$ of $D$ can be written (for $d, d' \geq 5$):
\[
\tilde{D} = \left( \begin{array}{cccc|ccccc}
0 & 1 & 1 & 1 		& 2 & \cdots & \cdots & \cdots & 2\\
1 & 0 & 1 & 1 		& 1 & \cdots & \cdots & \cdots & 1\\
1 & 1 & 0 & 2		& 2 & \cdots & \cdots & 2 & 1\\
1 & 1 & 2 & 0		& 1 & 2 & \cdots & \cdots & 2\\
\hline
1 & 2 & 1 & 2								& 3 & \cdots & \cdots & 3 & 2\\
\vdots & \vdots & 2 & \vdots 			& \vdots & & & & 3\\
\vdots & \vdots & \vdots & \vdots 	& \vdots & & & & \vdots \\
\vdots & \vdots & \vdots & 2 			& 3 & & & & \vdots \\
1 & 2 & 2 & 1 								& 2 & 3 & \cdots & \cdots & 3
\end{array} \right)
\]
Similarly we will write $\tilde{\xi}^t$ the relevant submatrix of the coupling $\xi$, and 
$\langle \xi^t , D \rangle =  \langle \tilde{\xi}^t , \tilde{D} \rangle = \sum_{i,j} \xi^t_{i,j} d(x_i,x_j)$.

When $d$ and $d'$ are large, the optimal coupling moves the mass mainly along 
the edge $(x,x')$. This gives a general formula for values of $d$ and $d'$.
For values of $d' \geq 6$, we set 
\[
\Delta = \tilde{D} + L_1 + 2 L_2 + L_3 + L_4 - C_1 - 2 C_2 - C_3 - C_4 - 2 C_5 
- 3 (C_6 + \cdots + C_{d'}) - 2 C_{d'+1}
\]
and show that $\Delta$ has only nonnegative coefficients.

\begin{itemize}

\item $d'\geq 6$ and $d \geq 5$:
\[ \arraycolsep=4pt
\Delta = \left( \begin{array}{cccc|ccccc} 
0 & 0 & 1 & 1   & 1 & 0 & \cdots & 0 & 1 \\
2 & 0 & 2 & 2   & 1 & \vdots & & \vdots & 1 \\
1 & 0 & 0 & 2   & 1 & \vdots & & \vdots & 0 \\
1 & 0 & 2 & 0   & 0 & 0 & \cdots & 0 & 1 \\
\hline
0 & 0 & 0 & 1   & 1 & 0 & \cdots & 0 & 0 \\
\vdots & \vdots & 1 & \vdots    & \vdots & \vdots & & \vdots & 1  \\
\vdots & \vdots & \vdots & \vdots    & \vdots & \vdots & & \vdots & \vdots \\
\vdots & \vdots & \vdots & 1    & 1 & \vdots & & \vdots & \vdots \\
0 & 0 & 1 & 0 		& 0 & 0 & \cdots & 0 & 1
\end{array} \right)
\!, 
\tilde{\xi}^t = \left( \begin{array}{cccc|ccccc}
\frac{t}{d'} & x_t & 0 & 0 	& 0 & \cdots & \cdots & \cdots & 0 \\
0 & \frac{t}{d'} & 0 & 0 		& 0 & y_t & \cdots & y_t & 0 \\
0 & 0 & \frac{t}{d'} & 0 		& 0 & y_t & \cdots & y_t & 0 \\
0 & 0 & 0 & \frac{t}{d'}		& 0 & y_t & \cdots & y_t & 0 \\
\hline
0 & \cdots & \cdots & 0		& 0 & y_t & \cdots & y_t & \frac{t}{d'} \\
\vdots &  &  & \vdots		& \vdots & z_t & \cdots & z_t & 0 \\
\vdots &  &  & \vdots		& \vdots & \vdots &  & \vdots & \vdots \\
\vdots &  &  & \vdots		& 0 & z_t & \cdots & z_t & \vdots \\
0 & \cdots & \cdots & 0		& \frac{t}{d'} & y_t & \cdots & y_t & 0
\end{array} \right)
\]
where 
\[
x_t = 1 - t - \frac{t}{d'}, \quad 
y_t = \frac{t}{d'-5} \left( \frac{1}{d} - \frac{1}{d'} \right), \quad
z_t = \frac{t}{d (d'-5)}
\] 
($x_t \ge 0$ whenever $t \le \frac{d'}{d'+1}$) and the cost is 
$W_1 = 1 + t \left(2- \frac{4}{d} - \frac{8}{d'}\right)$. 
We check easily that $\tilde\xi^t$ has nonnegative coefficients
and satisfies conditions~\eqref{eq:Kuhn-Tucker} and \eqref{eq:Kuhn-Tucker0}.


\item $d'\geq 6$ and $d =4$: \\
The distance matrix $D$ has only $5$ rows, and its submatrix is slightly different:
\[ \arraycolsep=4pt
\tilde{D} = \left( \begin{array}{cccc|ccccc}
0 & 1 & 1 & 1		& 2 & \cdots & \cdots & \cdots & 2 \\
1 & 0 & 1 & 1		& 1 & \cdots & \cdots & \cdots & 1 \\
1 & 1 & 0 & 2 		& 2 & \cdots & \cdots & 2 & 1 \\
1 & 1 & 2 & 0		& 1 & 2 & \cdots & \cdots & 2 \\
\hline
1 & 2 & 1 & 1		& 2 & 3 & \cdots & 3 & 2
\end{array} \right), \;
\Delta = \left( \begin{array}{cccc|ccccc}
0 & 0 & 1 & 1		& 1 & 0 & \cdots & 0 & 1 \\
2 & 0 & 2 & 2 		& 1 & 0 & \cdots & 0 & 1 \\
1 & 0 & 0 & 2		& 1 & 0 & \cdots & 0 & 0 \\
1 & 0 & 2 & 0		& 0 & 0 & \cdots & 0 & 1 \\
\hline
0 & 0 & 0 & 0		& 0 & \cdots & \cdots & \cdots & 0
\end{array} \right)
\]
and an optimal transportation plan is
\[
\tilde\xi^t = \left( \begin{array}{cccc|ccccc}
\frac{t}{d'} & x_t & 0 & 0			& 0 & \cdots & \cdots & \cdots & 0 \\
0 & \frac{t}{d'} & 0 & 0			& 0 & y_t & \cdots & y_t & 0 \\
0 & 0 & \frac{2 t}{3 d'} & 0		& 0 & z_t & \cdots & z_t & \frac{2 t}{3 d'} \\
0 & 0 & 0 & \frac{2 t}{3 d'}		& \frac{2 t}{3 d'} & z_t & \cdots & z_t & 0 \\
\hline
0 & 0 & \frac{t}{3 d'} & \frac{t}{3 d'}	& \frac{t}{3 d'} & z_t & \cdots & z_t & \frac{t}{3 d'}
\end{array} \right)
\]
where 
\[
x_t = 1 - t - \frac{t}{d'}, \quad 
y_t = \frac{t}{d'-5} \left( \frac{1}{4} - \frac{1}{d'} \right), \quad
z_t = \frac{t}{d'-5} \left( \frac{1}{4} - \frac{4}{3 d'} \right), \quad
t \leq \frac{d'}{d'+1}
\]
with cost $W_1 = 1 + t - \frac{8 t}{d'} = 1 + t \left(2- \frac{4}{4} - \frac{8}{d'}\right)$.
Note that, thanks to $d' \geq 6$, we have $\frac{1}{4} - \frac{4 }{3 d'} \geq 0$, 
so $z_t \geq 0$ as needed.


\item $d'\geq 6$ and $d =3$: \\
Similarly, we have 
\[
\tilde{D} = \left( \begin{array}{cccc|ccccc}
0 & 1 & 1 & 1		& 2 & \cdots & \cdots & \cdots & 2 \\
1 & 0 & 1 & 1		& 1 & \cdots & \cdots & \cdots & 1 \\
1 & 1 & 0 & 2		& 2 & \cdots & \cdots & 2 & 1 \\
1 & 1 & 2 & 0		& 1 & 2 & \cdots & \cdots & 2 \\
\end{array} \right) , \quad
\Delta = \left( \begin{array}{cccc|ccccc}
0 & 0 & 1 & 1		& 1 & 0 & \cdots & 0 & 1 \\
2 & 0 & 2 & 2 		& 1 & 0 & \cdots & 0 & 1 \\
1 & 0 & 0 & 2		& 1 & 0 & \cdots & 0 & 0 \\
1 & 0 & 2 & 0		& 0 & 0 & \cdots & 0 & 1 \\
\end{array} \right)
\]
\[
\tilde{\xi}^t = \left( \begin{array}{cccc|ccccc}
\frac{t}{d'} & x_t & 0 & 0			& 0 & \cdots & \cdots & \cdots & 0 \\
0 & \frac{t}{d'} & 0 & 0			& 0 & y_t & \cdots & y_t & 0 \\
0 & 0 & \frac{t}{d'} & 0			& 0 & z_t & \cdots & z_t & \frac{t}{d'} \\
0 & 0 & 0 & \frac{t}{d'}			& \frac{t}{d'} & z_t & \cdots & z_t & 0 
\end{array} \right)
\]
where
\[
x_t = 1 - t - \frac{t}{d'},\quad 
y_t = \frac{t}{d'-5} \left(\frac{1}{3} - \frac{1}{d'} \right), \quad 
z_t = \frac{t}{d'-5} \left( \frac{1}{3} - \frac{2}{d'} \right), \quad 
t \leq \frac{d'}{d'+1}
\]
and the cost is $W_1 = 1 + t \left(2- \frac{4}{3} - \frac{8}{d'}\right)$.
Thanks to $d' \geq 6$, we have $z_t \ge 0$ as needed.

\end{itemize}

We conclude that in all cases where $d' \ge 6$, we have 
$\ric = \frac{4}{d} + \frac{8}{d'} - 2$ as claimed.


\subsection{$\kappa^1$ computation}

In the case $t = 1$, the measure $\mu_x^1$ does not put any weight on $x$, hence any transfer plan
is identically zero along the line corresponding to $x$ (and along the column corresponding to $y$).
In our notations, it~would amount to discarding the first line and the second column of all matrices 
previously written, but for clarity and comparison purposes we will keep them, though they play no role.
Again we only deal with the cases of variable degree, and leave the remaining 
cases to the computer program.

\begin{itemize}
\item $d' \geq 6$ and $d \geq 6$ 
\\
The matrix $\Delta$ is the same as above, but because of the additional constraint,
an optimal coupling is now given by
\[
\tilde\xi^1 = \frac{1}{d d'} \left( \begin{array}{cccc|ccccc}	
0 & 0 & 0 & 0		& 0 & \cdots & \cdots & \cdots & 0 \\
0 & 0 & 0 & 0		& 0 & \frac{d'}{d'-5} & \cdots & \frac{d'}{d'-5} & 0 \\
0 & 0 & d & 0		& 0 & \frac{d' - d}{d'-5} & \cdots & \frac{d' - d}{d'-5} & 0 \\
0 & 0 & 0 & d		& 0 & \frac{d' - d}{d'-5} & \cdots & \frac{d' - d}{d'-5} & 0 \\
\hline
0 & 0 & \cdots & 0				& 0 & \frac{d' - d}{d'-5} & \cdots & \frac{d' - d}{d'-5} & d \\
d & \vdots &  & \vdots 			& 0 & \frac{d' - d}{d'-5} & \cdots & \frac{d' - d}{d'-5} & 0 \\
0 & \vdots &  & \vdots			& \vdots & \frac{d'}{d'-5} & \cdots & \frac{d'}{d'-5} & 0 \\
\vdots & \vdots &  & \vdots	& \vdots & \vdots &  & \vdots & \vdots \\
\vdots & \vdots &  & \vdots	& 0 & \frac{d'}{d'-5} & \cdots & \frac{d'}{d'-5} & \vdots \\
0 & 0 & \cdots & 0				& d & \frac{d' - d}{d'-5} & \cdots & \frac{d' - d}{d'-5} & 0
\end{array} \right)
\]
with cost $W_1 = \frac{1}{d d'} (- 4 d' - 8 d + 3 d d') = 3 - \frac{4}{d} - \frac{8}{d'}$.
Since we need at least eight distinct lines, we cannot compute $\kappa^1$ by this method anymore 
when $d \leq 5$.


\item $d = 5$ and $d' \geq 7$
\[
\tilde{D} = \left( \begin{array}{cccc|ccccc}
0 & 1 & 1 & 1 		& 2 & \cdots & \cdots & \cdots & 2\\
1 & 0 & 1 & 1 		& 1 & \cdots & \cdots & \cdots & 1\\
1 & 1 & 0 & 2		& 2 & \cdots & \cdots & 2 & 1\\
1 & 1 & 2 & 0		& 1 & 2 & \cdots & \cdots & 2\\
\hline
1 & 2 & 1 & 2								& 3 & \cdots & \cdots & 3 & 2\\
1 & 2 & 2 & 1 								& 2 & 3 & \cdots & \cdots & 3
\end{array} \right) , \quad
\Delta = \left( \begin{array}{cccc|ccccc}
0 & 0 & 1 & 1   & 1 & 0 & \cdots & 0 & 1 \\
2 & 0 & 2 & 2   & 1 & \vdots & & \vdots & 1 \\
1 & 0 & 0 & 2   & 1 & \vdots & & \vdots & 0 \\
1 & 0 & 2 & 0   & 0 & 0 & \cdots & 0 & 1 \\
\hline
0 & 0 & 0 & 1   & 1 & 0 & \cdots & 0 & 0 \\
0 & 0 & 1 & 0 		& 0 & 0 & \cdots & 0 & 1
\end{array} \right)
\]
An optimal coupling is given by
\[
\tilde\xi^1 = \frac{1}{5 d'} \left( \begin{array}{cccc|ccccc}	
0 & 0 & 0 & 0		& 0 & \cdots & \cdots & \cdots & 0 \\
0 & 0 & 0 & 0		& 0 & \frac{d'}{d'-5} & \cdots & \frac{d'}{d'-5} & 0 \\
0 & 0 & 2 & 0		& 0 & \frac{d' - 7}{d'-5} & \cdots & \frac{d' - 7}{d'-5} & 5 \\
0 & 0 & 0 & 5		& 1 & \frac{d' - 6}{d'-5} & \cdots & \frac{d' - 6}{d'-5} & 0 \\
\hline
3 & 0 & 3 & 0		& 0 & \frac{d' - 6}{d'-5} & \cdots & \frac{d' - 6}{d'-5} & 0 \\
2 & 0 & 0 & 0		& 4 & \frac{d' - 6}{d'-5} & \cdots & \frac{d' - 6}{d'-5} & 0
\end{array} \right)
\quad W_1 = \frac{11 d' - 40}{5d'} = 3 - \frac{4}{5} - \frac{8}{d'}
\]


\item $d = 4$ and $d' \geq 7$ 
\[
\tilde{D} = \left( \begin{array}{cccc|ccccc}
0 & 1 & 1 & 1 		& 2 & \cdots & \cdots & \cdots & 2\\
1 & 0 & 1 & 1 		& 1 & \cdots & \cdots & \cdots & 1\\
1 & 1 & 0 & 2		& 2 & \cdots & \cdots & 2 & 1\\
1 & 1 & 2 & 0		& 1 & 2 & \cdots & \cdots & 2\\
\hline
1 & 2 & 1 & 1		& 2 & 3 & \cdots & 3 & 2
\end{array} \right) , \quad
\Delta = \left( \begin{array}{cccc|ccccc}
0 & 0 & 1 & 1   & 1 & 0 & \cdots & 0 & 1 \\
2 & 0 & 2 & 2   & 1 & \vdots & & \vdots & 1 \\
1 & 0 & 0 & 2   & 1 & \vdots & & \vdots & 0 \\
1 & 0 & 2 & 0   & 0 & 0 & \cdots & 0 & 1 \\
\hline
0 & 0 & 0 & 0   & 1 & 0 & \cdots & 0 & 0 \\
\end{array} \right)
\]
An optimal coupling is given by
\[
\tilde{\xi}^1 = \frac{1}{4d'} \left( \begin{array}{cccc|ccccc}	
0 & 0 & 0 & 0		& 0 & \cdots & \cdots & \cdots & 0 \\
0 & 0 & 0 & 0		& 0 & \frac{d'}{d'-5} & \cdots & \frac{d'}{d'-5} & 0 \\
0 & 0 & 3 & 0		& 0 & \frac{d' - 7}{d'-5} & \cdots & \frac{d' - 7}{d'-5} & 4 \\
0 & 0 & 0 & 3		& 4 & \frac{d' - 7}{d'-5} & \cdots & \frac{d' - 7}{d'-5} & 0 \\
\hline
4 & 0 & 1 & 1		& 0 & \frac{d' - 6}{d'-5} & \cdots & \frac{d' - 6}{d'-5} & 0 \\
\end{array} \right) \quad
W_1 = \frac{1}{4d'}(8 d' - 32) = 3 - \frac{4}{4} - \frac{8}{d'}
\]


\item $d = 3$ and $d' \geq 8$ 
\[
\tilde{D} = \left( \begin{array}{cccc|ccccc}
0 & 1 & 1 & 1 		& 2 & \cdots & \cdots & \cdots & 2\\
1 & 0 & 1 & 1 		& 1 & \cdots & \cdots & \cdots & 1\\
1 & 1 & 0 & 2		& 2 & \cdots & \cdots & 2 & 1\\
1 & 1 & 2 & 0		& 1 & 2 & \cdots & \cdots & 2
\end{array} \right)
\]
In this case, the Lagrange multipliers are a bit different and
\[
\begin{array}{rcl}
\Delta &=& \tilde{D} + 2  L_1 + 2 L_2 + L_3 + L_4 
\\ 
&& - 2 C_1 - 2 C_2 - C_3 - C_4 - 2 C_5 - 3 (C_6 + \cdots +  C_{d'}) - 2 C_{d'+1}
\\
&=& 
\left( \begin{array}{cccc|ccccc}
  0 & 1 & 2 & 2 & 2 & 1 & \cdots & 1 & 2\\
  1 & 0 & 2 & 2 & 1 & 0 & \cdots & 0 & 1\\
  0 & 0 & 0 & 2 & 1 & 0 & \cdots & 0 & 0\\
  0 & 0 & 2 & 0 & 0 & 0 & \cdots & 0 & 1
\end{array} \right)
\end{array}
\]
An optimal coupling is given by
\[
\tilde{\xi}^1 = \frac{1}{3d'}
\left(
\begin{array}{cccc|ccccc}
0 & 0 & 0 & 0 		& 0 & \cdots & \cdots & \cdots & 0 \\
0 & 0 & 0 & 0 		& 0 & \frac{d'}{d'-5} & \cdots & \frac{d'}{d'-5} & 0 \\
2 & 0 & 3 & 0 		& 0 &\frac{d'-8}{d'-5} & \cdots & \frac{d'-8}{d'-5} & 3\\
1 & 0 & 0 & 3		& 3 & \frac{d'-7}{d'-5}& \cdots & \frac{d'-7}{d'-5} & 0
\end{array} \right)  \qquad
W_1 = \frac{1}{3d'}(5 d' - 21) = 3 - \frac{4}{3} - \frac{7}{d'}
\]
\end{itemize}


\subsection{The rectangular parallelepiped}  \label{ann:parallelepiped}

\tikzset{smallnode/.style={circle,inner sep=0pt,minimum size=5mm,draw=black,fill=white}}
\tikzset{fore/.style={ultra thick,draw=white,double=black,double distance=0.8pt}}

\begin{figure}[h] \centering
\begin{tikzpicture}
	\node [smallnode] (0) at (0, 0) {1};
	\node [smallnode] (1) at (2.5, 2) {2};
	\node [smallnode] (2) at (0.5, 2) {3};
	\node [smallnode] (3) at (-2, 0) {4};
	\node [smallnode] (4) at (0, -1) {5};
	\node [smallnode] (5) at (2.5, 1) {6};
	\node [smallnode,draw=gray!60] (6) at (0.5, 1) {7};
	\node [smallnode] (7) at (-2, -1) {8};
	\draw [style=outline] (1) to (2);
	\draw [style=outline] (2) to (3);
	\draw [style=outline] (3) to (7);
	\draw [style=outline] (7) to  node[below=2pt,fill=white] {$b$} (4);
	\draw [style=outline] (4) to node[below right=2pt] {$a$} (5) ;
	\draw [style=outline] (5) to (1);
	\draw [gray!60] (5) to (6);
	\draw [gray!60] (6) to (2);
	\draw [gray!60] (7) to (6);
	\draw [outline] (4) to node[left=1pt,fill=white] {$c$} (0) ;
	\draw [fore] (0) to (1); 
	\draw [fore] (3) to (0);
\end{tikzpicture}
\caption{Rectangular parallelepiped.}
\end{figure}
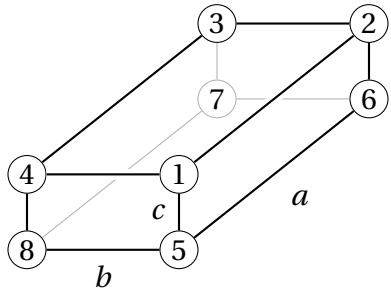

The (Delaunay) cotangent laplacian for the rectangular parallelepiped with edges of lengths 
$|e_{12}| = a$, $| e_{14} | = b$, $| e_{15} | = c$, is given by
\[ 
\Delta_{12} = \frac{b + c}{2 a} , \; \Delta_{14} = \frac{c + a}{2
   b}, \; \Delta_{15} = \frac{a + b}{2 c} \]
so that $\sum_{y \sim 1} d ( 1, y) \Delta_{1 y} = a + b + c$, hence
we normalize to $\mu_x = (1-t) \delta_x + t \dot{\mu}_x$ with
\[ 
\dot{\mu}_1 (2) = \frac{b + c}{2 a ( a + b + c)} 
= \frac{1}{2 a} - \frac{1}{2 ( a + b + c)}
,\;
\dot{\mu}_1 ( 1) = \frac{1}{2}  \left( \frac{1}{a} + \frac{1}{b} +
   \frac{1}{c} - \frac{3}{a + b + c} \right)
\]\[
\dot{\mu}_1 ( 4) = \frac{1}{2 b} - \frac{1}{2( a + b + c)}, \; 
\dot{\mu}_1 ( 5) = \frac{1}{2 c} - \frac{1}{2 ( a + b + c)} 
\]
The distance matrix is
\[ D = \left(\begin{array}{cccccc}
     0 & a & a + b & b & c & c + a\\
     a & 0 & b & a + b & c + a & c\\
     a + b & b & 0 & a & a + b + c & b + c\\
     b & a + b & a & 0 & b + c & a + b + c\\
     c & c + a & a + b + c & b + c & 0 & a\\
     c + a & c & b + c & a + b + c & a & 0
   \end{array}\right) \]
\begin{multline*} 
D - aL_1 + aC_1 + aC_4 + aC_5 - aL_4 - aL_5
\\
= \left( \begin{array}{cccccc}
     0 & 0 & b & b & c & c\\
     2 \hspace{0.25em} a & 0 & b & 2 \hspace{0.25em} a + b & 2 \hspace{0.25em}
     a + c & c\\
     2 \hspace{0.25em} a + b & b & 0 & 2 \hspace{0.25em} a & 2 \hspace{0.25em}
     a + b + c & b + c\\
     b & b & 0 & 0 & b + c & b + c\\
     c & c & b + c & b + c & 0 & 0\\
     2 \hspace{0.25em} a + c & c & b + c & 2 \hspace{0.25em} a + b + c & 2
     \hspace{0.25em} a & 0
   \end{array} \right)
\end{multline*}
A optimal coupling is
\[ 
\xi^t = \left( \begin{array}{cccccc}
     \frac{t}{2}  \left( \frac{1}{a} - \frac{1}{a + b + c} \right)  & 1 -
     \frac{t}{2}  \left( \frac{2}{a} + \frac{1}{b} + \frac{1}{c} - \frac{4}{a
     + b + c} \right) & 0 & 0 & 0 & 0\\
     0 & \frac{t}{2}  \left( \frac{1}{a} - \frac{1}{a + b + c} \right) & 0 & 0
     & 0 & 0\\
     0 & 0 & 0 & 0 & 0 & 0\\
     0 & 0 & \frac{t}{2}  \left( \frac{1}{b} - \frac{1}{a + b + c} \right) & 0
     & 0 & 0\\
     0 & 0 & 0 & 0 & 0 & \frac{t}{2}  \left( \frac{1}{c} - \frac{1}{a + b + c}
     \right)\\
     0 & 0 & 0 & 0 & 0 & 0
   \end{array} \right) \]
with cost $W_1 = \displaystyle a \left( 1 - t \left( \frac{1}{a} - \frac{1}{a + b + c}
\right) \right)$ and curvature $\displaystyle \ric = \frac{1}{a} - \frac{1}{a + b +
c}$ along an edge of length $a$ as claimed. 


\paragraph{Acknowledgements.}

The second author wishes to thank Djalil Chafa\"{i} for his stimulating remarks
and for pointing out the work of Ollivier.

%


\bibliographystyle{alpha}
\bibliography{lite}

\bigskip
\noindent
Beno\^{i}t Loisel \\
ENS de Lyon, 15 parvis Ren\'e Descartes - BP 7000 69342 Lyon Cedex 07, France \\
benoit.loisel@ens-lyon.fr
\\

\noindent
Pascal Romon \\
Universit\'e Paris-Est, LAMA (UMR 8050), F-77454, Marne-la-Vall\'ee, France \\
pascal.romon@u-pem.fr

\end{document}